\documentclass[10pt]{amsart}

\usepackage{amssymb,latexsym, mathtools,tikz}
\usepackage{enumerate}
\usepackage{graphicx}
\usepackage{float}
\usepackage{placeins}
\usepackage{mdframed}
\usepackage{amssymb}
\usepackage{esint}
\usepackage{cool}
\usepackage[all,cmtip]{xy}
\usepackage{mathtools}
\usepackage{amstext} 
\usepackage{array}   
\usepackage[shortlabels]{enumitem}
\usepackage{ytableau}

\newcolumntype{L}{>{$}l<{$}} 
\newcolumntype{C}{>{$}c<{$}}
\newtheorem{theorem}{Theorem}[section]

\newtheorem{cor}[theorem]{Corollary}
\newtheorem{prop}[theorem]{Proposition}
\newtheorem{setup}[theorem]{Setup}
\theoremstyle{definition}
\newtheorem{definition}[theorem]{Definition}

\newtheorem{obs}[theorem]{Observation}
\newtheorem{notation}[theorem]{Notation}

\theoremstyle{remark}
\newtheorem{remark}[theorem]{Remark}

\newtheorem{the context}[theorem]{The Context}
\newtheorem{question}[theorem]{Question}
\numberwithin{equation}{theorem}
\numberwithin{equation}{section}

\newcommand{\cat}[1]{\mathcal{#1}}

\newcommand{\rank}{\operatorname{rank}}

\newcommand{\coker}{\operatorname{Coker}}

\newcommand{\tor}{\operatorname{Tor}}
\newcommand{\im}{\operatorname{Im}}

\newcommand{\Ker}{\operatorname{Ker}}

\newcommand{\ideal}[1]{\mathfrak{#1}}
\newcommand{\m}{\ideal{m}}

\newcommand{\supp}{\operatorname{Supp}}

\renewcommand{\I}{\mathcal{I}}

\renewcommand{\geq}{\geqslant}
\renewcommand{\leq}{\leqslant}
\renewcommand{\ker}{\Ker}
\renewcommand{\hom}{\Hom}

\newcommand{\Hom}{\operatorname{Hom}}

\newcommand{\maps}[5]{\xymatrix{#1 \ar[r]^-{#3} & #2 \\
#4 \ar@{|->}[r] & #5 \\}}
\newcommand{\kos}{\textrm{Kos}}

\newcommand{\mfa}{\mathfrak{a}}

\def\w{\wedge}

\def\im{\operatorname{im}}

\newcommand{\sgn}{\operatorname{sgn}}

\newcommand{\rk}{\textrm{rk}}

\begin{document}
\title[Resolutions of Monomial Ideals]{Minimal Free Resolutions of Certain Equigenerated Monomial Ideals}



\author{Keller VandeBogert }
\date{\today}

\maketitle

\begin{abstract}
    Let $R = k[x_1, \dotsc , x_n]$ denote the standard graded polynomial ring over a field $k$. We study certain classes of equigenerated monomial ideals with the property that the so-called complementary ideal has no linear relations on the generators. We then use iterated trimming complexes to deduce Betti numbers for such ideals. Furthermore, using a result on splitting mapping cones by Miller and Rahmati, we construct the minimal free resolutions for all ideals under consideration explicitly and conclude with questions about extra structure on these complexes.
\end{abstract}

\section{Introduction}

Let $(R, \m , k)$ denote a local ring. The computation of minimal free resolutions of arbitrary ideals $I \subseteq R$ is a problem that remains open, even in relatively simple cases. In this paper, we consider instead the class of monomial ideals; that is, ideals minimally generated by monomials. Such ideals seem to exist at the intersection of commutative algebra and combinatorics, and are hence the subject of a large body of research.

In \cite{taylor1966ideals}, Taylor constructed what is now called the Taylor resolution. This complex, aside from being a free resolution for any monomial ideal $I$, also possesses many other desirable properties. For instance, it always admits the structure of an associative differential graded (DG) algebra, and is cellular (see \cite{bayer1998cellular}). In general, however, this resolution is highly nonminimal. Moreover, in a now classic paper (see \cite{reisner1976cohen}), Reisner initiated the study the squarefree monomial ideals through an associated simplicial complex (the \emph{Stanley-Reisner} complex), and showed that the Betti numbers of such ideals can be computed as simplicial homology of link subcomplexes of the Stanley-Reisner complex. In theory, this gives a method of computing the Betti numbers of an arbitrary monomial ideal (since polarization will reduce to the squarefree case), but would of course require a closed form for all such aforementioned simplicial homologies.

Kaplansky posed the problem of describing the minimal free resolution of a monomial ideal in a polynomial ring. In general, this has turned out to be a very difficult problem. A large class of ideals for which an explicit minimal free resolution can be constructed is for so-called Borel-fixed (or stable) ideals. This resolution was constructed in \cite{eliahou1990minimal} and is now called the Eliahou-Kervaire resolution; it is a special case of the more general iterated mapping cone procedure. This resolution, similar to the Taylor resolution, admits the structure of a DG algebra (see \cite{peeva19960}) and is cellular (see \cite{mermin2010eliahou}). Likewise, a squarefree analogue of the Eliahou-Kervaire resolution is considered in \cite{aramova1998squarefree} and generalized in \cite{gasharov2002resolutions}, for which many of the properties of the standard Eliahou-Kervaire resolution remain valid.

Monomial ideals are a class of ideals for which combinatorial techniques have also proved very effective for the computation of such minimal free resolutions. One can reduce the study of arbitrary monomial ideals to the study of squarefree monomial ideals via polarization; once this reduction is made, there is a standard one-to-one correspondence between squarefree monomial ideals $I \subseteq k[x_1 , \dotsc , x_n]$ and simplicial complexes $\Delta$ on $n$ vertices. This perspective was introduced in \cite{miller2004combinatorial} and is used to deduce homological information of a monomial ideal $I$ based on the combinatorial data of $\Delta$. An excellent survey of this perspective, along with a collection of the literature on the topic, may be found in \cite{ha2006resolutions}.

Even more recently, the problem of a general minimal free resolution for all monomial ideals has been attacked in \cite{eagon2019minimal}. This fascinating construction relies heavily on extensive combinatorial machinery; as a result of its generality, the complex itself is not simple to construct, but has the advantage of being described almost entirely in a combinatorial fashion.

In this paper, we restrict ourselves to the case of equigenerated monomial ideals; that is, ideals generated in a single degree. A na\"ive method of obtaining such ideals is to start with the ideal generated by all monomials of degree $d$, $(x_1 , \dotsc , x_n)^d \subset k[x_1 , \dotsc , x_n]$, and then delete some of the generators. The graded minimal free resolution of $(x_1 , \dotsc , x_n)^d$ is well known (see Proposition \ref{prop:resnofpower}), and so one would only need machinery for which the Betti numbers after deleting generators could be deduced. This machinery is provided by so-called trimming complexes as in \cite{vandebogert2021trimming}.

These complexes have previously been used to resolve homogeneous grade $3$ ideals $I \subset k[x,y,z]$ defining compressed rings with socle $k(-s)^\ell \oplus k(-2s+1)$  ($\ell \geq 1$) in \cite{vandebogert2021resolution}, in which case these complexes are generically minimal. They have also been used to deduce explicit Betti numbers for certain classes of determinantal facet ideals in \cite{vandebogert2021trimming}. In this paper, we use these complexes to deduce Betti numbers and explicit minimal free resolutions of certain classes of equigenerated monomial ideals. In these particular cases, the minimal free resolutions obtained are quite simple to describe, and are computed without the use of any combinatorial tools. 

The paper is organized as follows. Section \ref{sec:introductionStuff} introduces necessary background, conventions, and definitions. In particular, we first recall the construction of the previously mentioned trimming complexes to be used for computing the desired Betti numbers. We then introduce background of Schur and Specht modules and two standard resolutions for both powers of complete intersections and the ideal generated by all squarefree monomials of a given degree in some polynomial ring. In Section \ref{sec:qmapsforschurandspecht} we build so-called $q_i$-maps for the aforementioned complexes to be used in the construction of trimming complexes.

In Sections \ref{sec:qmapsfortheLandSpechtcx} and \ref{sec:removingmultiplegens}, we compute explicit Betti tables for certain classes of equigenerated monomial ideals. In particular, we produce a large class of (squarefree) equigenerated monomial ideals with linear resolution. The definition of these ideals is phrased in terms of its so-called \emph{complementary ideal} (see Definition \ref{def:monomIdealDefs}. More precisely, we impose the condition that there are \emph{no} linear syzygies on the complementary ideal; in this case, certain maps associated to the complexes introduced in Section \ref{sec:introductionStuff} become much simpler.

Finally, in Section \ref{sec:linearstrandandres}, we use a result of Miller and Rahmati (see \cite{miller2018free}) about splitting mapping cones to compute explicit minimal free resolutions for the ideals of Section \ref{sec:removingmultiplegens} (see Theorem \ref{thm:theMinlResn}). In the case where these complexes are linear, the minimal free resolution is even simpler to describe: it is constructed as the kernel of a certain morphism of complexes (see Corollary \ref{cor:resnfordeletingpurepowers} and Theorem \ref{thm:resofsqfreetrimmed}).

\section{Trimming Complexes and Resolutions Arising from Schur/Specht modules}\label{sec:introductionStuff}

In this section, we first introduce iterated trimming complexes as constructed in \cite{vandebogert2021trimming}. The main takeaway of the trimming complexes is that a free resolution of certain ideals may be obtained by a simple mapping cone construction; in particular, this can be used to construct a free resolution of the ideal obtained by deleting some minimal generator from another ideal. Next, we give a brief overview of the construction of the $L$-complexes of Buchsbaum and Eisenbud and a complex constructed by Galetto. The building blocks for these complexes are Schur and Specht modules, respectively, and we recall some standard facts about these objects.

\subsection{Iterated Trimming Complexes}\label{sec:trimmingcx}

All proofs of the following results may be found in Section $2$ and $3$ of \cite{vandebogert2021trimming}. This first setup is needed in the construction of \emph{trimming complexes}

\begin{setup}\label{setup4}
Let $R=k[x_1,\dots, x_n]$ be a standard graded polynomial ring over a field $k$. Let $I \subseteq R$ be a homogeneous ideal and $(F_\bullet, d_\bullet)$ denote a homogeneous free resolution of $R/I$. 

Write $F_1 = F_1' \oplus \Big( \bigoplus_{i=1}^m Re_0^i \Big)$, where each $e^i_0$ generates a free direct summand of $F_1$. Using the isomorphism
$$\hom_R (F_2 , F_1 ) = \hom_R (F_2,F_1') \oplus \Big( \bigoplus_{i=1}^m \hom_R (F_2 , Re^i_0) \Big)$$
write $d_2 = d_2' + d_0^1 + \cdots + d^m_0$, where $d_2' \in \hom_R (F_2,F_1')$, $d^i_0 \in \hom_R (F_2 , Re^i_0)$. Let $\mfa_i$ denote any homogeneous ideal with
$$d^i_0 (F_2) \subseteq \mfa_i e^i_0,$$
and $(G^i_\bullet , m^i_\bullet)$ be a homogeneous free resolution of $R/\mfa_i$. 

Use the notation $K' := \im (d_1|_{F_1'} : F_1' \to R)$, $K^i_0 := \im (d_1|_{Re^i_0} : Re^i_0 \to R)$, and let $J := K' + \mfa_1 \cdot K^1_0+ \cdots + \mfa_m \cdot K_0^m$.
\end{setup}

\begin{prop}\label{prop:it1stmap}
Adopt notation and hypotheses of Setup \ref{setup4}. Then for each $i=1,\dots,m$ there exist maps $q^i_1 : F_2 \to G^i_1$ such that the following diagram commutes:
$$\xymatrix{& F_2 \ar[dl]_-{q^i_1} \ar[d]^{{d^i_0}'} \\
G_1 \ar[r]_{m_1} & \mfa \\},$$
where ${d^i_0}' : F_2 \to R$ is the composition
$$\xymatrix{F_2 \ar[r]^{d^i_0} & Re^i_0 \ar[r] & R\\},$$
the second map sending $e^i_0 \mapsto 1$. 
\end{prop}

\begin{prop}\label{prop:ittheyexist}
Adopt notation and hypotheses as in Setup \ref{setup4}. Then for each $i=1, \dots , m$ there exist maps $q^i_k : F_{k+1} \to G^i_{k}$ for all $k \geq 2$ such that the following diagram commutes:
$$\xymatrix{F_{k+1} \ar[d]_{q^i_k} \ar[r]^-{d_{k+1}} & F_k \ar[d]^{q^i_{k-1}} \\
G^i_k \ar[r]_{m^i_k} & G^i_{k-1} \\}$$
\end{prop}
The main result of this subsection is the following:

\begin{theorem}\label{thm:itres}
Adopt notation and hypotheses as in Setup \ref{setup4}. Then the mapping cone of the morphism of complexes
\begin{equation}\label{itcomx}
\xymatrix{\cdots \ar[r]^{d_{k+1}} &  F_{k} \ar[ddd]^{\begin{pmatrix} q_{k-1}^1 \\
\vdots \\
q_{k-1}^m \\
\end{pmatrix}}\ar[r]^{d_{k}} & \cdots \ar[r]^{d_3} & F_2 \ar[rrrr]^{d_2'} \ar[ddd]^{\begin{pmatrix} q_1^1 \\
\vdots \\
q_1^m \\
\end{pmatrix}} &&&& F_1' \ar[ddd]^{d_1} \\
&&&&&&& \\
&&&&&&& \\
\cdots \ar[r]^-{\bigoplus m^i_k} & \bigoplus_{i=1}^m G^i_{k-1} \ar[r]^-{\bigoplus m^i_{k-1}} & \cdots \ar[r]^-{\bigoplus m^i_2} & \bigoplus_{i=1}^m G^i_1 \ar[rrrr]^-{-\sum_{i=1}^\ell m^i_1(-)\cdot d_1(e^i_0)} &&&& R \\}\end{equation}
is acyclic and forms a free resolution of $R/J$. 
\end{theorem}

\begin{definition}\label{def:trimmingCxDef}
The \emph{iterated trimming complex} associated to the data of Setup \ref{setup4} is the mapping cone of the morphism of complexes of Theorem \ref{thm:itres}.
\end{definition}

 As an immediate consequence, one obtains a consequence for computing (graded) Betti numbers:

\begin{cor}\label{cor:ittorrk}
Adopt notation and hypotheses of Setup \ref{setup4}. Assume furthermore that the complexes $F_\bullet$ and $G_\bullet$ are minimal. Then for $i \geq 2$,
$$\dim_k \tor_i^R (R/J , k) = \rank F_i + \sum_{j=1}^m  \rank G^j_i - \rank \Bigg( \begin{pmatrix} q_i^1 \\
\vdots \\
q_i^m \\
\end{pmatrix} \otimes k\Bigg) - \rank \Bigg( \begin{pmatrix} q_{i-1}^1 \\
\vdots \\
q_{i-1}^m \\
\end{pmatrix} \otimes k\Bigg).$$
Similarly,
$$\mu (J) = \mu(K) -m + \sum_{j=1}^m \mu(\mfa_j) -\rank \Bigg( \begin{pmatrix} q_1^1 \\
\vdots \\
q_1^m \\
\end{pmatrix} \otimes k\Bigg) .\qquad \qquad \square $$ 
\end{cor}

\begin{remark}
Observe that one may restrict to homogeneous pieces in Corollary \ref{cor:ittorrk} to obtain graded Betti numbers as well.
\end{remark}

\subsection{L-Complexes And Resolutions of Squarefree Monomials}\label{sec:Lcompandresofsqfree}

The material up until Proposition \ref{prop:resnofpower}, along with proofs, can be found in \cite{buchsbaum1975generic} or Section $2$ of \cite{el2014artinian}. This first setup will be needed for the construction of the $L$-complexes of Buchsbaum and Eisenbud.

\begin{setup}\label{set:Lcomplexsetup}
Let $F$ denote a free $R$-module of rank $n$, and $S = S(F)$ the symmetric algebra on $F$ with the standard grading. Define a complex
$$\xymatrix{\cdots \ar[r] & \bigwedge^{a+1} F \otimes_R S_{b-1} \ar[r]^-{\kappa_{a+1,b-1}} & \bigwedge^{a} F \otimes_R S_{b} \ar[r]^-{\kappa_{a,b}} & \cdots}$$
where the maps $\kappa_{a,b}$ are defined as the composition
\begin{equation*}
    \begin{split}
         \bigwedge^{a} F \otimes_R S_{b} &\to \bigwedge^{a-1} F \otimes_R F \otimes_R S_{b} \\
         & \to \bigwedge^{a-1} F \otimes_R S_{b+1}
    \end{split}
\end{equation*}
where the first map is comultiplication in the exterior algebra and the second map is the standard module action (where we identify $F = S_1 (F)$). Define
$$L_b^a (F) := \ker \kappa_{a,b}.$$
Let $\psi: F \to R$ be a morphism of $R$-modules with $\im (\psi)$ an ideal of grade $n$. Let $\kos^\psi : \bigwedge^i F \to \bigwedge^{i-1} F$ denote the standard Koszul differential; that is, the composition
\begin{equation*}
    \begin{split}
        \bigwedge^i F &\to F \otimes_R \bigwedge^{i-1} F  \quad \textrm{(comultiplication)} \\
        &\to \bigwedge^{i-1} F \quad \textrm{(module action)} \\
    \end{split}
\end{equation*}
\end{setup}

\begin{definition}\label{def:Lcomplexes}
Adopt notation and hypotheses of Setup \ref{set:Lcomplexsetup}. Define the complex
\begin{equation*}
    \begin{split}
        &L(\psi , b) : \xymatrix{0 \ar[r] & L_b^{n-1} \ar[rr]^-{\kos^\psi \otimes 1} & & \cdots \ar[rr]^{\kos^\psi \otimes 1} & & L_b^0 \ar[r]^-{S_b (\psi)} & R \ar[r] & 0 } \\
    \end{split}
\end{equation*}
where $\kos^\psi \otimes 1  : L_b^a (F) \to L_b^{a-1}$ is induced by making the following diagram commute:
$$\xymatrix{\bigwedge^a F\otimes S_b (F) \ar[rr]^-{\kos^\psi \otimes 1}  & & \bigwedge^{a-1} F \otimes S_b(F)  \\
L_b^a (F) \ar[rr]^-{\kos^\psi \otimes 1} \ar[u] & & L_b^{a-1} (F) \ar[u] \\}$$
\end{definition}

\begin{prop}\label{prop:resnofpower}
Let $\psi: F \to R$ be a map from a free module $F$ of rank $n$ such that the image  $\im (\psi)$ is a grade $n$ ideal. Then the complex $L(\psi ,b)$ of Definition \ref{def:Lcomplexes} is a minimal free resolution of $R/\im (\psi)^b$
\end{prop}
We also have (see Proposition $2.5 (c)$ of \cite{buchsbaum1975generic})
\begin{equation*}
    \begin{split}
        &\rank_R L_b^a (F) = \binom{n+b-1}{a+b} \binom{a+b-1}{a}. \\
    \end{split}
\end{equation*}
Moreover, using the notation and language of Chapter $2$ of \cite{weyman2003}, $L_b^a (F)$ is the Schur module $L_{(a+1,1^{b-1})} (F)$. This allows us to identify a standard basis for such modules.

\begin{notation}
We use the English convention for partition diagrams. That is, the partition $(3,2,2)$ corresponds to the diagram
$$ 
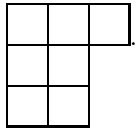
\begin{ytableau}
 \ & \ & \ \\
\ & \ \\
\ & \ \\
\end{ytableau}.$$
A Young tableau is standard if it is strictly increasing in both the columns and rows. It is semistandard if it is strictly increasing in the columns and nondecreasing in the rows.
\end{notation}

\begin{prop}\label{prop:standardbasis}
Adopt notation and hypotheses as in Setup \ref{set:Lcomplexsetup}. Then a basis for $L_b^a (F)$ is represented by all Young tableaux of the form
$$\ytableausetup
{boxsize=2em}
\begin{ytableau}
i_0 & j_1 &\cdots & j_{b-1} \\
i_1 \\
\vdots \\
i_a \\
\end{ytableau}$$
with $i_0 < \cdots < i_{a}$ and $i_0 \leq j_1 \leq \cdots \leq j_{b-1}$. 
\end{prop}

\begin{proof}
See Proposition $2.1.4$ of \cite{weyman2003} for a more general statement.
\end{proof}

\begin{remark}\label{rk:rkonstdbasis}
Adopt notation and hypotheses of Setup \ref{set:Lcomplexsetup}. Let $F$ have basis $f_1, \dotsc, f_n$. In the statement of Proposition \ref{prop:standardbasis}, we think of the tableau as representing the element
$$\kappa_{a+1,b-1} (f_{i_1} \w \cdots \w f_{i_{a+1}} \otimes f_{j_1} \cdots f_{j_{b-1}}) \in \bigwedge^a F \otimes S_b (F).$$
We will often write $f_{i_1} \w \cdots \w f_{i_{a+1}} \otimes f_{j_1} \cdots f_{j_{b-1}} \in L_b^a (F)$, with the understanding that we are identifying $L_b^a (F)$ with the cokernel of $\kappa_{a+2,b-2} : \bigwedge^{a+2} F \otimes S_{b-2} (F) \to \bigwedge^{a+1} F \otimes S_{b-1} (F)$.
\end{remark}

Next, we give a brief introduction of Specht modules and define the complex constructed by Galetto in \cite{galetto2020ideal}. The construction of Specht modules used here may be considered the dual construction, as in $7.4$ of \cite{fulton1997young}. Instead of the more standard presentation using row tabloids, the Specht modules here are constructed as the quotient of all column tabloids by the so-called straightening relations. 

\begin{definition}\label{def:spechtmodules}
Let $\lambda$ be a partition and $k$ a field. A column tabloid $[T]$ is an equivalence class of a tableau $T$ modulo alternating columns.

Let $M^\lambda$ denote the formal span of all column tabloids of shape $\lambda$. Define the map $\pi_{j,k} : M^\lambda \to M^\lambda$ by sending $[T] \mapsto \sum [S]$, where the sum is over all tableau $S$ obtained from $T$ by exchanging the top $k$ elements of the $(j+1)$st column with the $k$ elements in the $j$th column of $T$, while preserving the vertical order of each set of $k$ elements.

Let $\mu = \lambda^t$ denote the transpose partition. Then the maps $\pi_{j,k}$ are defined for $1 \leq j \leq \lambda_1-1$, $1 \leq k \leq \mu_{j+1}$. Define the submodule $Q^\lambda \subset M^\lambda$ to be the subspace spanned by all elements of the form
$$[T] - \pi_{j,k} ([T]),$$
where $j, \ k$ vary as above.

Then, with notation as above, define the Specht module $S^{\lambda}$ to be the quotient $M^\lambda / Q^\lambda$.
\end{definition}

\begin{definition}\label{def:sqfreecomplexpieces}
Let $d \leq n$ be integers. Define
\begingroup\allowdisplaybreaks
\begin{align*}
    U_i^{d,n} &= \textrm{Ind}_{S_{d+i} \times S_{n-d-i}}^{S_n} \big( S^{(d,1^i)} \otimes S^{(n-d-i)} \big) \\
    &= \bigoplus_{\sigma} \sigma \big( S^{(d,1^i)} \otimes S^{(n-d-i)} \big)
\end{align*}
\endgroup
where the direct sum is taken over all coset representatives for $S_{d+i} \times S_{n-d-i}$. 
\end{definition}

\begin{definition}\label{def:freddifferentials}
Let $R= k[x_1 , \dotsc , x_n]$ where $k$ is a field. Let $1 \leq d \leq n$ and $1 \leq i \leq n-d+1$. Define
$$F_i^{d,n} := U_{i-1}^{d,n} \otimes_k R(-d-i+1),$$
where $U_i^{d,n}$ is as in Definition \ref{def:sqfreecomplexpieces}. Given any Tableau $T$, define the differential
$$\partial_i^{d,n} ( [ T ] ) := \sum_{j=0}^i (-1)^{i-j} x_{a_j} [ T \backslash a_j ],$$
where
$$T:= \ytableausetup
{boxsize=2em}
\begin{ytableau}
a_1 & b_1 & \cdots & b_{d-1} \\
a_2 \\
\vdots \\
a_{i} \\
\end{ytableau},$$
and $i>1$. When $i=1$, define
$$\partial_1^{d,n} \Bigg( \ytableausetup{boxsize=2em}\begin{ytableau}
a_1 & b_1 & \cdots & b_{d-1} \\
\end{ytableau} \Bigg) = x_{a_1} x_{b_1} \cdots x_{b_{d-1}}.$$
Let $F_\bullet^{d,n}$ denote the complex
$$\xymatrix{0 \ar[r] & F_{n-d+1}^{d,n} \ar[rr]^-{\partial_{n-d+1}^{d,n}} &  & \cdots  \ar[r]^{\partial_2^{d,n}} & F_1^{d,n} \ar[r]^-{\partial_1^{d,n}} & R \\}.$$
\end{definition}

\begin{theorem}[\cite{galetto2020ideal}, Theorem 4.11]\label{thm:fredsres}
Let $n$ and $d$ be integers with $1 \leq d \leq n$. Then the complex $F_\bullet^{d,n}$ of Definition \ref{def:freddifferentials} is a $S_n$-equivariant minimal free resolution of quotient ring defined by the ideal generated by all squarefree monomials of degree $d$ in $R$.
\end{theorem}

\begin{notation}\label{notation:simplifyingTableaux}
Adopt notation as in Definition \ref{def:freddifferentials}. To the tabloid $[T]$ we will associate a formal basis element
    $$[T] \xleftrightarrow[]{} f_{a_1} \w f_{a_2} \w \cdots \w f_{a_i} \otimes f_{b_1} \cdot f_{b_2} \cdots f_{b_{d-1}},$$
    where the notation is meant to mimic the notation used for the modules $L_d^i$. This should cause no confusion, since the straightening relations/tabloid properties are directly compatible with the straightening relations for $L_d^i$ and the exterior/symmetric algebra relations.
\end{notation}

\section{$q_i$ Maps for Certain Schur and Specht Modules}\label{sec:qmapsforschurandspecht}

In this section, we construct the maps of Proposition \ref{prop:ittheyexist} in the case where the relevant modules are Schur and Specht modules, and they are being mapped to a Koszul complex. These maps are essential for the rest of the paper, as they are the building blocks employed for the iterated trimming complex construction. At the end of this section, we also take the opportunity to compute certain colon ideals; these colons will be used in later sections in order to count rank and deduce higher strands appearing in the minimal free resolutions of the ideals of interest.

\begin{notation}
Let $R$ be a commutative ring. Let $F$ be a free $R$-module of rank $n$ with basis $f_1 , \dotsc , f_n$ and let $\ell$, $b$ be integers. Fix indexing sets $J=(j_1 , \dotsc, j_\ell)$ with $j_1 < \cdots < j_\ell$ and $\alpha = (\alpha_1 , \dotsc , \alpha_n)$ with $\alpha_i \geq 0$ for each $i=1 , \dotsc , n$, $|\alpha| = b$.

The notation $f_J$ denotes $f_{j_1} \w \cdots \w f_{j_\ell} \in \bigwedge^\ell F$, the notation $f^J$ denotes $f_{j_1} \cdots f_{j_\ell} \in S_\ell (F)$, and the notation $f^\alpha$ denotes $f_1^{\alpha_1} \cdots f_n^{\alpha_n} \in S_b (F)$.
\end{notation}

\begin{definition}\label{def:generalqmaps}
Let $R$ be a commutative ring. Let $F$ be a free $R$-module of rank $n$ with basis $f_1 , \dotsc , f_n$ and let $\ell$, $b$ be integers. Fix indexing sets $J=(j_1 , \dotsc, j_\ell)$ with $j_1 < \cdots < j_\ell$ and $\alpha = (\alpha_1 , \dotsc , \alpha_n)$ with $\alpha_i \geq 0$ for each $i=1 , \dotsc , n$, $|\alpha| = b$. Define the maps $\phi_i^{J,\alpha} : \bigwedge^i F \otimes S_b (F) \to \bigwedge^i F$ via
$$\phi_i^{J,\alpha} ( f_I \otimes f^\beta ) = \begin{cases}
f_I & \textrm{if} \ I \subseteq J \ \textrm{and} \ \beta = \alpha \\
0 & \textrm{otherwise} \\
\end{cases}$$
\end{definition}

\begin{obs}
Adopt notation and hypotheses as in Definition \ref{def:generalqmaps}. Let $\psi: F \to R$ be a homomorphism of $R$-modules, and $\kos^\psi : \bigwedge^i F \to \bigwedge^{i-1} F$ the induced Koszul differential. Then the following diagram commutes:
$$\xymatrix{\bigwedge^i F \otimes S_b(F) \ar[d]^-{\phi_i^{J,\alpha}} \ar[r]^-{\kos^\psi \otimes 1} & \bigwedge^{i-1} F \otimes S_b (F) \ar[d]^-{\phi_{i-1}^{J,\alpha}} \\
\bigwedge^i F \ar[r]^{\kos^\psi} & \bigwedge^{i-1} F. \\}$$
Moreover, for all $i \geq 1$, $\phi_i^{J,\alpha}$ induces the commutative diagram
$$\xymatrix{L^i_b (F) \ar[d]^-{\phi_i^{J,\alpha}} \ar[r]^-{\kos^\psi \otimes 1} & L^{i-1}_b (F) \ar[d]^-{\phi_{i-1}^{J,\alpha}} \\
\bigwedge^i F \ar[r]^{\kos^\psi} & \bigwedge^{i-1} F, \\}$$
where $L^i_b (F)$ is as in Setup \ref{set:Lcomplexsetup}. More precisely, this map is realized as:
$$\phi_i^{J , \alpha} (\kappa_{i+1,b-1} (f_I \otimes f^{\beta})) = \begin{cases}
\sgn (i) f_{I \backslash i} & \textrm{if} \ i \in I, \ \beta+ \epsilon_i = \alpha \\
0 & \textrm{otherwise}
\end{cases}$$
\end{obs}

\begin{definition}\label{def:spechtqmaps}
Let $R$ be a commutative ring. Let $F$ be a free $R$-module of rank $m$ with basis $f_1 , \dotsc , f_m$ and let $\ell$, $d$ be integers. Fix indexing sets $J=(j_1 , \dotsc, j_\ell)$ with $j_1 < \cdots < j_\ell$ and $I=(i_1 , \dotsc , i_d)$ with $i_1 < \cdots < i_d$. Let  $\psi : F \to R$ be an $R$-module homomorphism and $$T:= \ytableausetup
{boxsize=2em}
\begin{ytableau}
a_0 & b_1 & \cdots & b_{d-1} \\
a_1 \\
\vdots \\
a_{\ell} \\
\end{ytableau}$$
a standard tableau with $a_0 < \cdots < a_\ell$ and $b_1 < \cdots < b_{d-1}$. Define maps 
$$\psi_\ell^{J,I} : S^{(d,1^\ell)} \to \bigwedge^\ell F$$
on the equivalence class of the column tabloid $[T ] \in S^{(d, 1^{\ell})}$ by setting
$$\psi_\ell^{J,I} \big( [ T ] \big) :=  \begin{cases}
\sgn (a_i) f_{\{a_0 , \dots , \widehat{a_i}, \dots , a_\ell \}} & \textrm{if} \ I=\{ b_1 , \dotsc , b_{d-1} \} \cup \{ a_i \} \ \textrm{for some} \ 0 \leq i \leq \ell \\
& \textrm{and} \ \{a_0 , \dots , \widehat{a_i}, \dots , a_\ell \} \subseteq J, \\
0 & \textrm{otherwise}. \\
\end{cases}$$
Observe that this is well defined since the above definition is compatible with the shuffling relations on $S^{(d , 1^\ell)}$. Moreover, extending by linearity, this induces a map
$$\phi_\ell^{J, I } : F^{d,n}_\ell \to \bigwedge^\ell U$$
making the following diagram commute:
$$\xymatrix{F^{d,n}_\ell \ar[d]^-{\phi_\ell^{J,\alpha}} \ar[r]^-{\partial^{d,n}_\ell} & F^{d,n}_{\ell-1} \ar[d]^{\phi_{\ell-1}^{J,\alpha}} \\
\bigwedge^{\ell} F \ar[r]^-{\kos^\psi} & \bigwedge^{\ell-1} F, \\}$$
where $F^{d,n}_\ell$ and $\partial^{d,n}_\ell$ are as in Definition \ref{def:freddifferentials} and $\kos^\psi$ denotes the induced Koszul differential.
\end{definition}

\begin{prop}\label{prop:idealgensbyrows}
Adopt notation and hypotheses as in Setup \ref{setup4}, and assume that $d_0^i (F_2) = \mfa_i e_0^i$. Then the ideals $\mfa_i \subseteq R$ do not depend on the choice of differential $d_2$.
\end{prop}

\begin{proof}
Assume for simplicity that $m=1$. Then we will prove a slightly stronger statement; namely, $\mfa_1 = (K' : K_0^1)$. The containment $\mfa_1 \subseteq (K' : K_0^1)$ is trivial, so let $r \in (K':K_0^1)$. Assume $\rank F_1' = f'$ and let $e_1 , \dotsc , e_{f'}$ denote a basis for $F_1'$. 

By definition, there exist elements $r_i \in R$ such that
$$r_1 d_1(e_1) + \cdots + r_{f'} d_1(e_{f'}) = rd_1(e_0^1),$$
$$\implies d_1 ( r_1 e_1 + \cdots + r_{f'} e_{f'} - r e_0^1 ) = 0.$$
However, by the assumption on $\mfa_1$, this implies $r \in \mfa_1$ as desired.
\end{proof}

\begin{notation}
Let $R = k[x_1 , \dotsc , x_n]$, where $k$ is any field. If $\alpha = (\alpha_1 , \dotsc , \alpha_n)$, then the notation $x^\alpha$ denotes $x_1^{\alpha_1} \cdots x_n^{\alpha_n}$. Given such an $\alpha$, define $|\alpha| := \alpha_1 + \cdots + \alpha_n$. If $J = \{ j_1 < \cdots < j_n\}$, then the notation $x^J$ will denote $x_{j_1} \cdots x_{j_n}$. Given such a $J$, define $|J| = n$, the cardinality of $J$.

The notation $\epsilon_i$ will denote the vector with a $1$ in the $i$th entry and $0$'s elsewhere.
\end{notation}
The following Propositions are immediate.

\begin{prop}\label{prop:maxlidealrows}
Let $R = k[x_1 , \dotsc , x_n]$ where $k$ is any field and let $\alpha = (\alpha_1 , \dotsc , \alpha_n)$ be an exponent vector with $|\alpha| = d$. If $K' := ( x^{\beta} \mid |\beta|=d, \  \beta \neq \alpha )$, then
$$(K' : x^\alpha) = \begin{cases}
(x_1 , \dotsc , \widehat{x_i} , \dotsc , x_n) & \textrm{if} \ \alpha = d \epsilon_i \ \textrm{for some} \ 1 \leq i \leq n \\
(x_1 , \dotsc , x_n ) & \textrm{otherwise}. \\
\end{cases}$$
\end{prop}

\begin{prop}\label{prop:sqfreeidealrows}
Let $R = k[x_1 , \dotsc , x_n]$ where $k$ is any field and let $J = \{ j_1 < \cdots < j_d \}$. If $K' := ( x^I \mid |I|=d, \ I \neq J)$, then
$$(K' : x^J) = (x_i \mid i \notin J).$$
\end{prop}

\section{$q_i$ Maps for the Complexes $L ( \psi , b)$ and $F^{n,m}_\bullet$}\label{sec:qmapsfortheLandSpechtcx}

We can now use the maps constructed in Section \ref{sec:qmapsforschurandspecht} to find the Betti tables for resolving certain subsets of the standard generating sets for powers of the maximal ideal and all squarefree monomials of a given degree. Our first goal is to compute the ranks of the maps $\phi_\ell^{J , \alpha}$ of Definition \ref{def:generalqmaps} and $\psi_\ell^{J,I}$ of Definition \ref{def:spechtmodules}. We begin with some definitions and notation related to monomial ideals which will be in play for the rest of the paper.

\begin{definition}\label{def:monomIdealDefs}
Let $R = k[x_1 , \dots , x_n]$ be a standard graded polynomial ring over a field $k$. Let $K$ denote an equigenerated monomial ideal with generators in degree $d$. Define 
$$G(K) := \textrm{unique minimal generating set of} \ K \ \textrm{consisting of monic monomials.}$$
Given a monomial ideal $K$, define the (squarefree) complementary ideal $\overline{K}$ to be the ideal with minimal generating set:
$$G(\overline{K}) = \begin{cases}
    \{ \textrm{degree} \ d \ \textrm{squarefree monomials} \} \backslash G(K) & \textrm{if} \ K \ \textrm{squarefree}, \\
    \{ \textrm{degree} \ d \ \textrm{monomials} \} \backslash G(K) & \textrm{otherwise.} \\
    \end{cases}$$
\end{definition}

The following setup will be used for constructing the Betti table/minimal free resolution when the monomial ideals of interest are not squarefree.

\begin{setup}\label{set:trimmingmaxideal}
Let $R=k[x_1 , \dotsc , x_n]$ where $k$ is a field and let $F = \bigoplus_{i=1}^n Re_i$ be a free module of rank $n$ with map $\psi : F \to R$ sending $e_i \mapsto x_i$. Let $d \geq 1$ denote any integer and $L(\psi , d)$ the complex of Definition \ref{def:Lcomplexes}. Fix an exponent vector $\alpha= (\alpha_1 , \dotsc , \alpha_n)$ with $|\alpha|=d$. Let 
$$U = \begin{cases} 
\bigoplus_{j\neq i} Re_j & \textrm{if} \ \alpha = d \epsilon_i \\
F & \textrm{otherwise},\\
\end{cases}$$ 
with map $\psi : U \to R$ defined by sending $e_j \mapsto x_j$. 

Let $\phi^{I,\alpha}_\ell : L_d^\ell (F) \to \bigwedge^\ell U$ for $1 \leq \ell \leq n$ be the maps of Definition \ref{def:generalqmaps}, where 
$$I = \begin{cases}
[n] \backslash \{i \} & \textrm{if} \ \alpha = d \epsilon_i \\
[n] & \textrm{otherwise}. \\
\end{cases}$$
\end{setup}
The following notation will be convenient in many of the ensuing computations:
\begin{notation}\label{not:nalphanotation}
Adopt notation and hypotheses of Setup \ref{set:trimmingmaxideal}. Let $\supp (\alpha) = \{ i \mid \alpha_i >0\}$ and define
$$n_\alpha := |\supp (\alpha)|$$
\end{notation}

\begin{prop}\label{prop:qrankforpurepower}
Adopt notation and hypotheses of Setup \ref{set:trimmingmaxideal} with $\alpha = d \epsilon_i$ for some $1 \leq i \leq n$. The maps $\phi^{I,\alpha}_\ell : L_d^\ell (F) \to \bigwedge^\ell U$ are surjective for all $1 \leq \ell \leq n-1$. In particular,
$$\rank (\phi^{I,\alpha}_\ell \otimes k) = \binom{n-1}{\ell}$$
\end{prop}

\begin{proof}
Let $J \subset I$ with $J = (j_1 , \dotsc , j_\ell)$. It suffices to show that $e_J$ is in the image of $\phi_\ell^{I,d\epsilon_i}$ for any choice of $J$. Order the set $J \cup \{ i \}$ so that
$$j_1 < \cdots < j_k < i < j_{k+1} < \cdots < j_\ell.$$
This is possible since $i \notin J$ by construction of the free module $U$. Then, by definition,
$$\phi_{\ell}^{I,d \epsilon_i} ( e_{J \cup \{ i \}} \otimes e_i^{d-1} ) = \sgn (i) e_J.$$
\end{proof}

\begin{cor}\label{cor:btablepurepower}
Adopt notation and hypotheses of Setup \ref{set:trimmingmaxideal} with $\alpha = d \epsilon_i$ for some $1 \leq i \leq n$. Let $K'$ be an equigenerated momomial ideal with $\overline{K'} = (x_i^d)$. Then, $R/K'$ has Betti table
\begin{equation*}\begin{tabular}{C|C C C C C C C C C}
     & 0 & 1 & \cdots &  \ell & \cdots &  n  \\
     \hline 
   0  & 1 & \cdot & \cdots &  \cdot &\cdots &  \cdot \\
   
   \vdots & \cdot & \cdot & \cdots  &\cdot  & \cdots & \cdot  \\
   
   d-1 & \cdot & \binom{n+d-1}{d}-1 & \cdots & \binom{n+d-1}{\ell+d} \binom{d+\ell-2}{\ell-1} -\binom{n-1}{\ell-1} & \cdots & \binom{n+d-2}{n-1}-1 \\
\end{tabular}
\end{equation*}
In particular, $R/K'$ has projective dimension $n$ with linear resolution and defines a ring of type $\binom{n+d-2}{n-1}-1$.
\end{cor}

\begin{prop}\label{prop:maxpowerqranks}
Adopt notation and hypotheses of Setup \ref{set:trimmingmaxideal}. Then the maps $\phi^{I,\alpha}_\ell : L_d^\ell (F) \to \bigwedge^\ell U$ are such that
$$\rank (\phi^{I,\alpha}_\ell \otimes k) =\binom{n}{\ell} - \binom{n - n_\alpha}{\ell - n_\alpha},$$
for all $1 \leq \ell \leq n$.
\end{prop}

\begin{proof}
We shall enumerate a subset of bases whose images under $\phi^{I, \alpha}_\ell$ form a linearly independent set, then show that the image of any other standard basis element lies in the image spanned by this set. Counting the size of this set will then yield the rank.

To this end, enumerate the set $\{ i \mid \alpha_i > 0 \} = \{ k_1 , \dotsc , k_{n_\alpha} \}$, where $k_1 < \cdots < k_{n_\alpha} $. Consider the set $S$ consisting of all standard basis elements of the form
$$e_{\{k_1 , \dotsc, k_s \} \cup J'} \otimes e^{\alpha - \epsilon_{k_s} },$$
in $L_\ell^d (F)$ with $s \leq n_\alpha$, $|J'|=\ell-s+1$. By definition,
$$\phi^{I, \alpha}_\ell (e_{\{k_1 , \dotsc, k_s \} \cup J'} \otimes e^{\alpha - \epsilon_{k_s} }) = \begin{cases}
\sgn(k_s) e_{J'} & \textrm{if} \ s=1 \\
\sgn(k_s) e_{\{k_1 , \dotsc , k_{s-1} \} \cup J'} & \textrm{otherwise}. \\
\end{cases}$$
The collection of all basis elements as above, where $1 \leq s \leq n_\alpha$, is evidently a linearly independent set since it is an irredundant subset of a basis for $\bigwedge^\ell U$. 

Let $1 \leq r \leq n_\alpha$ and consider any standard basis element of the form $e_{ \{k_r \} \cup J'} \otimes e^{\alpha - \epsilon_{k_r}}$. Let $t := \min \{ s \mid k_s \notin J \}$. Assume first that $t>1$; by definition of $t$, $\{ k_1 , \dotsc , k_{t-1} \} \subseteq J'$, so we may write $J' = \{k_1 , \dotsc , k_{t-1} \} \cup J''$ for some $J''$. Then,
$$\phi^{I,\alpha}_\ell (\sgn (k_r) e_{ \{k_r \} \cup J'} \otimes e^{\alpha - \epsilon_{k_r}} ) = \phi^{I,\alpha}_\ell ( - \sgn(k_t) e_{\{k_1 , \dotsc , k_t \} \cup J'' } \otimes e^{\alpha - \epsilon_{k_t}}),$$
and the element on the right is the image of an element of $S$. Likewise, if $t=1$, then
$$\phi^{I,\alpha}_\ell ( \sgn(k_r) e_{ \{k_r \} \cup J'} \otimes e^{\alpha - \epsilon_{k_r}} ) = \phi^{I,\alpha}_\ell ( -\sgn(k_1) e_{\{k_1 \} \cup J' } \otimes e^{\alpha - \epsilon_{k_1}}),$$
and again the element on the right is the image of an element of $S$. Thus, counting the cardinality of $S$, we see that this is counting all possible indexing sets $J'$ with $|J'| = \ell-s+1$ and $J' \cap \{ k_1 , \dotsc , k_s \} = \varnothing$, for $1 \leq s \leq n_\alpha$. It is a trivial counting exercise to see
$$|S| = \sum_{i=1}^{n_{\alpha}} \binom{n-i}{\ell-i+1}= \sum_{i=1}^{n_{\alpha}} \binom{n-i}{n-\ell -1},$$
and one can moreover check that
$$\sum_{i=1}^{n_{\alpha}} \binom{n-i}{n-\ell -1} = \binom{n}{\ell} - \binom{n - n_\alpha}{\ell - n_\alpha}.$$
\end{proof}
The following is an immediate result of Proposition \ref{prop:maxpowerqranks} combined with Corollary \ref{cor:ittorrk}.

\begin{cor}\label{cor:maxlidealbtable}
Adopt notation and hypotheses as in Setup \ref{set:trimmingmaxideal}. Let $K'$ be an equigenerated momomial ideal with $\overline{K'} = (x^\alpha)$. Then, $R/K'$ has Betti table
\begin{equation*}\begin{tabular}{C|C C C C C C C C C}
     & 0 & 1 & \cdots &  \ell & \cdots &  n  \\
     \hline 
   0  & 1 & \cdot & \cdots & \cdot &\cdots &  \cdot \\
   
   \vdots & \cdot  & \cdot & \cdots &\cdot &\cdots  & \cdot \\
   
   d-1 & \cdot & \binom{n+d-1}{d}-1 & \cdots & \binom{n+d-1}{\ell+d} \binom{d+\ell-2}{\ell-1} -\binom{n}{\ell-1} + \binom{n-n_\alpha}{\ell-1-n_\alpha} & \cdots & \binom{n+d-2}{n-1}-n_\alpha \\
   
   d & \cdot & \cdot & \cdots & \binom{n-n_\alpha}{\ell - n_\alpha} & \cdots & 1 \\
\end{tabular}
\end{equation*}
\end{cor}
The following result in the case of an Artinian ideal is a statement about the non-cyclicity of the associated inverse system; this behavior is highly dependent on the chosen generating set. For instance, choosing instead the generating set to be the maximal minors of the associated Sylvester matrix for $(x_1 , \dotsc , x_n)^2$, it is not hard to see that removing the generator $x_1 x_n$ will yield a grade $n$ Gorenstein ideal for all $n \geq 2$.

\begin{cor}
Adopt notation and hypotheses of Setup \ref{set:trimmingmaxideal}.  Let $K'$ be an equigenerated momomial ideal generated in degree $d \geq2$ with $\overline{K'} = (x^\alpha)$. Then, $R/K'$ is Gorenstein if and only if $n=d=2$, in which case $K' = (x_1^2, x_2^2)$. 
\end{cor}

\begin{proof}
By Gorenstein duality, it is immediate that if $K'$ is Gorenstein, then $d=2$. This implies that for any choice of $\alpha$, $n_{\alpha} \leq 2$. Moreover, using the Betti table of Corollary \ref{cor:maxlidealbtable}, $K'$ defines a ring of type $n-n_{\alpha} + 1 \geq n-1$, whence $n=2$. 
\end{proof}

Next, we adopt the following setup. This setup is the squarefree analog of Setup \ref{set:trimmingmaxideal}, and will instead be used to compute the Betti table/minimal free resolution when the ideals of interest are squarefree.

\begin{setup}\label{set:trimmingsqfree}
Let $R=k[x_1 , \dotsc , x_n]$ where $k$ is a field and let $F_\bullet^{d,n}$ denote the complex of Definition \ref{def:freddifferentials}. Fix an indexing set $I= (i_1 , \dotsc , i_d)$ and let $U = \bigoplus_{j \notin I} Re_j$ with map $\psi : U \to R$ defined by sending $e_j \mapsto x_j$. 

Let $\psi^{I,I^c}_\ell : F^{d,n}_\ell \to \bigwedge^\ell U$ for $1 \leq \ell \leq n-d$ be the maps of Definition \ref{def:spechtqmaps}, where $I^c = [n] \backslash I$.
\end{setup}

\begin{prop}\label{prop:spechtqranks}
Adopt notation and hypotheses as in Setup \ref{set:trimmingsqfree}. The maps $\psi^{I,I^c}_\ell : F^{d,n}_\ell \to \bigwedge^\ell U$ are surjective for all $1 \leq \ell \leq n-d$. In particular,
$$\rank (\psi^{I,I^c}_\ell \otimes k) = \binom{n-d}{\ell}$$
\end{prop}

\begin{proof}
Let $J \subset I^c$ be any indexing set with $J = (j_1 , \dotsc , j_\ell)$. It suffices to show that the basis element $e_J \in \bigwedge^\ell U$ is in the image of $\psi_\ell^{I , I^c}$. 

Order the set $J \cup \{ i_1 \}$, so that
$$j_1 < \cdots < j_{k} < i_1  < j_{k+1} < \cdots < j_\ell$$
for some $k < \ell$. Then, observe that the hook tableau with $J \cup \{ i_1 \}$ ordered appropriately in the first column and $(i_2 , \dotsc , i_d)$ along the first row has image $\sgn (i_1) e_J$.
\end{proof}

\begin{cor}\label{cor:sqfreebettitable}
Adopt notation and hypotheses as in Setup \ref{set:trimmingsqfree}.  Let $K'$ be a squarefree equigenerated momomial ideal with $\overline{K'} = (x^I)$. Then, $R/K'$
\begin{equation*}\begin{tabular}{C|C C C C C C C C C}
     & 0 & 1 & \cdots &  \ell & \cdots & n-d+1  \\
     \hline 
   0  & 1 & \cdot & \cdots &  \cdot &\cdots &  \cdot \\
   
   \vdots &\cdot  & \cdot & \cdots & \cdot & \cdots & \cdot \\
   
   d-1 & \cdot & \binom{n}{d}-1 & \cdots & \binom{d+\ell-2}{\ell-1} \binom{n}{d+\ell-1} -\binom{n-d}{\ell-1} & \cdots &  \binom{n-1}{n-d} - 1 \\
\end{tabular}
\end{equation*}
In particular, $R/K'$ has projective dimension $n-d+1$ with linear resolution and defines a ring of type $\binom{n-1}{n-d} - 1$.
\end{cor}

\section{Betti Tables for Certain Classes of Equigenerated Monomial Ideals}\label{sec:removingmultiplegens}

This section is an iterated version of Section \ref{sec:qmapsfortheLandSpechtcx}; that is, we consider the \emph{iterated} trimming complex associated to the maps constructed in the previous section. It turns out that under sufficient hypotheses, these maps stay well-behvaed when removing multiple generators at a time. The following setup is similar to Setup \ref{set:trimmingmaxideal}, but with more data to keep track of:

\begin{setup}\label{set:trimmingmaxidealiter}
Let $R=k[x_1 , \dotsc , x_n]$ where $k$ is a field and let $F = \bigoplus_{i=1}^n Re_i$ be a free module of rank $n$ with map $\psi : F \to R$ sending $e_i \mapsto x_i$. Let $d \geq 1$ denote any integer and $L(\psi , d)$ the complex of Definition \ref{def:Lcomplexes}. Fix exponent vectors $\alpha^s= (\alpha^s_1 , \dotsc , \alpha^s_n)$ with $|\alpha^s|=d$ for $1 \leq s \leq r$. Assume that for all $s \neq t$, $\deg \textrm{lcm} (x^{\alpha^s} , x^{\alpha^t}) \geq d+2$. Let 
$$U_s = \begin{cases} 
\bigoplus_{j\neq i} Re_j & \textrm{if} \ \alpha^s = d \epsilon_i \\
F & \textrm{otherwise},\\
\end{cases}$$ 
with map $\psi : U_s \to R$ induced by sending $e_j \mapsto x_j$. 

Let $\phi^{I_s,\alpha^s}_\ell : L_d^\ell (F) \to \bigwedge^\ell U$ for $1 \leq \ell \leq n$ be the maps of Definition \ref{def:generalqmaps}, where 
$$I_s = \begin{cases}
[n] \backslash \{i \} & \textrm{if} \ \alpha^s = d \epsilon_i \\
[n] & \textrm{otherwise}. \\
\end{cases}$$
\end{setup}

\begin{obs}\label{obs:canuseittrimm}
Adopt notation and hypotheses as in Setup \ref{set:trimmingmaxidealiter}. Let $K'$ be an equigenerated momomial ideal with $\overline{K'} = (x^{\alpha^1}, \dots , x^{\alpha^r})$ and let
$$\mfa_s := \begin{cases}
(x_1 , \dotsc , \widehat{x_i} , \dotsc , x_n) & \textrm{if} \ \alpha^s = d \epsilon_i \\
(x_1 , \dotsc , x_n) & \textrm{otherwise} \\
\end{cases}.$$
Then $\mfa_s x^{\alpha^s} \subseteq K'$ for all $1 \leq s \leq r$.
\end{obs}

\begin{proof}
Suppose for sake of contradiction that the containment $\mfa_t x^{\alpha^t} \not\subset K'$ for some $1 \leq t \leq r$. Let
$$K_t := ( x^\beta \mid |\beta|=d, \ \beta \neq \alpha^t \}, $$
and observe that $(K_t : x^{\alpha^t} ) = \mfa_t$ by Proposition \ref{prop:maxlidealrows}. This means that for some $s \neq t$, $x_i \cdot x^{\alpha^s} = x_j \cdot x^{\alpha^t}$, contradicting the LCM hypothesis on each $\alpha^s$.  
\end{proof}

\begin{remark}
In the notation of the statement of Observation \ref{obs:canuseittrimm}, this is saying that the construction of Theorem \ref{thm:itres} applied to the ideals $\mfa_s$, for $1 \leq s \leq r$, yields a resolution of $R/K'$. 
\end{remark}
The following Proposition makes precise the previously mentioned fact that the maps of Definition \ref{def:generalqmaps} are ``well-behaved" when removing multiple generators.

\begin{prop}\label{prop:nonearenonzero}
Adopt notation and hypotheses as in Setup \ref{set:trimmingmaxidealiter}. Enumerate the set $\supp (\alpha) = \{ k^s_1 , \dotsc , k^s_{n_{\alpha^s}} \}$ with $k^s_1 < \cdots < k^s_{n_{\alpha^s}}$. Then for all $t \neq s$ and $p \leq n_{\alpha^s}$,
$$\phi^{I^t,\alpha^t} (e_{\{k_1^s , \dotsc , k_p^s \} \cup J' } \otimes e^{\alpha^s - \epsilon_{k_p^s}} ) = 0.$$
\end{prop}

\begin{proof}
Suppose for sake of contradiction that there exists some $t \neq s$ and $1 \leq p \leq n_{\alpha^s}$ such that
$$\phi^{I^t,\alpha^t} (e_{\{k_1^s , \dotsc , k_p^s \} \cup J' } \otimes e^{\alpha^s - \epsilon_{k_p^s}} ) \neq 0.$$
This is possible if and only if there exists $q \in \{k_1^s , \dotsc , k_p^s \} \cup J'$ such that $\alpha^t = \alpha^s - \epsilon_{k_p^s} + \epsilon_q$. This implies that $\alpha^t - \alpha^s = \epsilon_q - \epsilon_{k_p^s}$, which is a clear contradiction to the LCM hypothesis on each $\alpha^s$.
\end{proof}

\begin{cor}\label{cor:stackedqranksmaxideal}
Adopt notation and hypotheses as in Setup \ref{set:trimmingmaxidealiter}. Then,
$$\rank \Bigg( \begin{pmatrix}
\phi^{I_1,\alpha^1}_\ell \\
\phi^{I_2,\alpha^2}_\ell \\
\vdots \\
\phi^{I_{r},\alpha^{r}}_\ell \\
\end{pmatrix} \otimes k  \Bigg)=  \sum_{s=1}^{r} \rank (\phi^{I_s,\alpha^s}_\ell \otimes k).$$
\end{cor}

\begin{cor}\label{cor:btablemaxidealiterated}
Adopt notation and hypotheses as in Setup \ref{set:trimmingmaxidealiter}. Define $\rk_\ell := \sum_{s=1}^{r} \rank (\phi^{I_s,\alpha^s}_\ell \otimes k)$. Let $K'$ be an equigenerated momomial ideal with $\overline{K'} = (x^{\alpha^1}, \dots , x^{\alpha^r})$. Then $R/ K'$ has Betti table
\begin{equation*}\begin{tabular}{C|C C C C C C C C C}
     & 0 & 1 & \cdots &  \ell & \cdots &  n  \\
     \hline 
   0  & 1 & \cdot & \cdots &  \cdot &\cdots & \cdot \\
   
   \vdots & \cdot & \cdot & \cdots & \cdot & \cdots & \cdot \\
   
   d-1 & \cdot & \binom{n+d-1}{d}-r & \cdots & \binom{n+d-1}{\ell+d} \binom{d+\ell-2}{\ell-1} -\rk_{\ell-1} & \cdots & \binom{n+d-2}{n-1}-\sum_{s=1}^r n_{\alpha^s} \\
   
   d & \cdot & \cdot & \cdots & \sum_{s=1}^r \rank \bigwedge^\ell U_s - \rk_\ell & \cdots & r \\
\end{tabular}
\end{equation*}
\end{cor}

As a special case of the above, we can compute the Betti table of an equigenerated monomial ideal whose complementary ideal consists only of pure powers.

\begin{cor}\label{cor:btablepurepowers}
Adopt notation and hypotheses as in Setup \ref{set:trimmingmaxidealiter} and let $B = \{ k_1 < \cdots < k_r \}$. Let $K'$ be an equigenerated momomial ideal with $\overline{K'} = (x_{k_1}^d, \dots , x_{k_r}^d)$. Then $R/ K'$ has Betti table
\begin{equation*}\begin{tabular}{C|C C C C C C C C C}
     & 0 & 1 & \cdots &  \ell & \cdots &  n-d+1  \\
     \hline 
   0  & 1 & \cdot & \cdots & \cdot &\cdots &  \cdot \\
   
   \dots & \cdot & \cdot & \cdots & \cdot & \cdots & \cdot  \\
   
   d-1 & \cdot & \binom{n+d-1}{d}-r & \cdots & \binom{n+d-1}{\ell+d} \binom{d+\ell-2}{\ell-1} -r\binom{n-1}{\ell-1} & \cdots & \binom{n+d-2}{n-1}-r \\
\end{tabular}
\end{equation*}
In particular, $R/ K'$ has projective dimension $n$ with linear resolution and defines a ring of type $\binom{n+d-2}{n-1}-r$.
\end{cor}

The rest of this section is just the squarefree analog of the first half of this section. It turns out that the squarefree case is, in some sense, much simpler than the non-squarefree case. We will see that these ideals \emph{always} have a linear minimal free resolution. We will first need to adopt the following setup, which the reader should take as the squarefree analog of Setup \ref{set:trimmingmaxidealiter}.

\begin{setup}\label{set:ittrimmingsqfree}
Let $R=k[x_1 , \dotsc , x_n]$ where $k$ is a field and let $F_\bullet^{d,n}$ denote the complex of Definition \ref{def:freddifferentials}. Fix indexing sets $I_j= (i_{j1} , \dotsc , i_{jd})$ for $1 \leq j \leq r$ with the property that $|I_j \cap I_i| \leq d-2$ for all $i \neq j$. Let $U_j= \bigoplus_{\ell \notin I_j} Re_\ell$ with map $\psi : U_j \to R$ defined by sending $e_\ell \mapsto x_\ell$. 

Let $\psi^{I_j,I_j^c}_\ell : F^{d,n}_\ell \to \bigwedge^\ell U_j$ for $1 \leq \ell \leq n-d$ be the maps of Definition \ref{def:spechtqmaps}, where $I_j^c = [n] \backslash I_j$.
\end{setup}

Observe that the proof of the following is essentially identical to that of Observation \ref{obs:canuseittrimm}, where we employ Proposition \ref{prop:sqfreeidealrows} instead.

\begin{obs}\label{obs:canuseittrimmsqfree}
Adopt notation and hypotheses as in Setup \ref{set:ittrimmingsqfree}. Let $K'$ be a squarefree equigenerated momomial ideal with $\overline{K'} = (x^{I_1}, \dots , x^{I_r})$
and let
$$\mfa_s := (x_j \mid j \notin I_s) \quad (1 \leq s \leq r).$$
Then $\mfa_s x^{I_s} \subseteq K'$ for all $1 \leq s \leq r$.
\end{obs}
In a similar manner, the proof of the following Proposition is essentially identical to that of Proposition \ref{prop:nonearenonzero}.

\begin{prop}
Adopt notation and hypotheses as in Setup \ref{set:ittrimmingsqfree}. Then for all $t \neq s$ and $p \leq d$,
$$\psi^{I_t,I_t^c} (e_{\{i_{1s} , \dotsc , i_{ps} \} \cup J' } \otimes e^{I_s - \epsilon_{i_{ps}}} ) = 0.$$
\end{prop}

\begin{cor}\label{cor:stackedqranksmaxidealsqfree}
Adopt notation and hypotheses as in Setup \ref{set:ittrimmingsqfree}. Then,
$$\rank \Bigg( \begin{pmatrix}
\psi^{I_1,I_1^c}_\ell \\
\psi^{I_2,I_2^c}_\ell \\
\vdots \\
\psi^{I_{r},I_r^c}_\ell \\
\end{pmatrix} \otimes k  \Bigg)=  \sum_{s=1}^{r} \rank (\psi^{I_s,I_s^c}_\ell \otimes k).$$
\end{cor}

\begin{cor}\label{cor:sqfreesubsetbtable}
Adopt notation and hypotheses as in Setup \ref{set:ittrimmingsqfree}. Let $K'$ be a squarefree equigenerated momomial ideal with $\overline{K'} = (x^{I_1}, \dots , x^{I_r})$. Then $R/K'$ has Betti table
\begin{equation*}\begin{tabular}{C|C C C C C C C C C}
     & 0 & 1 & \cdots &  \ell & \cdots & n-d+1  \\
     \hline 
   0  & 1 & \cdot & \cdots &  \cdot &\cdots &  \cdot \\
   
   \vdots & \cdot & \cdot & \cdots & \cdot & \cdots & \cdot \\
   
   d-1 & \cdot & \binom{n}{d}-r & \cdots & \binom{d+\ell-2}{\ell-1} \binom{n}{d+\ell-1} -r\binom{n-d}{\ell-1} & \cdots &  \binom{n-1}{n-d} - r \\
\end{tabular}
\end{equation*}
In particular, $R/K'$ has projective dimension $n-d+1$ with linear resolution and defines a ring of type $\binom{n-1}{n-d} - r$.
\end{cor}

\section{Explicit Minimal Free Resolutions}\label{sec:linearstrandandres}

In this section we produce the explicit minimal free resolutions of all of the ideals considered in Section \ref{sec:removingmultiplegens}. In particular, for the cases where the resolutions were linear, these resolutions may be obtained by simply taking the kernel of the morphisms of complexes constructed in the previous sections. The proofs of these results are based on the following more general theorem, which describes how to extract ``minimal" summands of mapping cones of complexes when the associated morphism of complexes is split. This first result is a specialized version of a result by Miller and Rahmati (see \cite[Proposition 2.1]{miller2018free})

\begin{theorem}\label{thm:splitMappingCone}
Consider the morphism of complexes
\begin{equation}\label{eq:coneDiagram}
    \xymatrix{\cdots \ar[r]^-{d_{k+1}} & F_{k} \ar[d]^-{q_k} \ar[r]^-{d_{k}} & \cdots \ar[r]^{d_2}& F_1 \ar[d]^-{q_1} \ar[r]^-{d_1} & F_0 \ar[d]^-{d_0} \\
\cdots \ar[r]^-{m_{k+1}} & G_k \ar[r]^-{m_k} & \cdots \ar[r]^-{m_2} & G_1 \ar[r]^-{m_1} & R. \\}
\end{equation}
For each $k >0$, let
$$A_k := \ker q_k, \ C_k := \coker q_k, \ \textrm{and} \ B_k := \im q_k,$$
and assume that the short exact sequences
$$0 \to A_k \to F_k \to B_k \to 0, \ \textrm{and}$$
$$0 \to B_k \to G_k \to C_k \to 0$$
are split, with $C_1 = 0$. Then the mapping cone of \ref{eq:coneDiagram} is the direct sum of a split exact complex and the following complex:
$$\cdots \to \begin{matrix} A_{k-1} \\ \oplus \\ C_k \end{matrix} \xrightarrow{\ell_k} \begin{matrix} A_{k-2} \\ \oplus \\ C_{k-1} \end{matrix} \xrightarrow{\ell_{k-1}} \cdots \xrightarrow{ \ \ell_3 \ } \begin{matrix} A_{1} \\ \oplus \\ C_{2} \end{matrix} \xrightarrow{ \ \ell_{2} \ } F_0 \xrightarrow{ \ d_0  \ } R,$$
where
\begingroup\allowdisplaybreaks
\begin{align*}
    \ell_k &:= \begin{pmatrix}
d_{k-1} & \Theta_k \\
0 & - m_k \\
\end{pmatrix}, \quad (k \geq 3), \\
\ell_2 &:= \begin{pmatrix}
d_{1} & \Theta_2 \\
\end{pmatrix}, 
\end{align*}
\endgroup

and $\Theta_k: C_k \to A_{k-2}$ is the composition
\begin{equation*}
\begin{split}
    C_k &\xrightarrow{\textrm{inclusion}} G_k \\
    &\xrightarrow{m_k} G_{k-1} \\
    &\xrightarrow{\textrm{projection}} B_{k-1} \\
    &\xrightarrow{\textrm{inclusion}} F_{k-1} \\
    &\xrightarrow{d_{k-1}} F_{k-2} \\
    &\xrightarrow{\textrm{projection}} A_{k-2} \\
    \end{split}
\end{equation*}
\end{theorem}

\begin{remark}
In the statement of Theorem \ref{thm:splitMappingCone}, it is understood that the differentials $d_k$ and $m_k$ appearing in the matrix form of $\ell_k$ are the maps induced by restricting to the subcomplex/quotient complex $A_\bullet$ and $C_\bullet$, respectively.
\end{remark}
In the next few definitions/results, we will be constructing the constituent building blocks of the minimal free resolution of the ideals of interest in Theorem \ref{thm:theMinlResn}.

\begin{definition}\label{def:Lsubmodules}
Adopt notation and hypotheses as in Setup \ref{set:trimmingmaxidealiter}, and let $B=\{ \alpha^1 , \dotsc , \alpha^r \}$. For each $s$, write $\supp (\alpha^s) = \{ k_1^s < \dots < k_{n_{\alpha^s}}^s \}$. For each $i>0$, define the free submodule $L_d^{i,B} (F) \subseteq L_d^i (F)$ to be generated by the following collection of basis elements, denoted $\cat{S}$ (all terms appearing are assumed to be standard basis elements as in Remark \ref{rk:rkonstdbasis}):
{\small
$$\begin{cases}
e_J \otimes e^\beta & \textrm{if} \ \beta \neq \alpha^s - \epsilon_{k_i^s} \ \textrm{for some} \ i, \\
e_J \otimes e^{\alpha - \epsilon_{k_p^s}} & \textrm{if} \ k_p^s \notin J, \\ 
\sgn (k_p^s) e_{J \cup \{ k_p^s \}} \otimes e^{\alpha - \epsilon_{k_p^s}} + \sgn(k_t^s) e_{\{k_1^s , \dotsc , k_t^s \} \cup J'} \otimes e^{\alpha - \epsilon_{k_t^s}} & \textrm{if} \ J= \{k_1^s , \dotsc , k_{t-1}^s \} \cup J' \\
&\textrm{for some} \ J', \  k_t^s \notin J \\
\end{cases}$$}
for all $1 \leq s \leq r$, $1 \leq p \leq n_{\alpha^s}$, where $J = (j_0 < \cdots < j_i)$.
\end{definition}

\begin{remark}\label{rk:submodsSimple}
In the case that $\alpha^s = d \epsilon_{i_s}$ for some indicies $i_1 < \dots < i_r$, the submodules $L^{i,B}_d$ as in Definition \ref{def:Lsubmodules} are obtained by simply deleting all standard basis elements of the form
$$e_{\{ i_s \} \cup J} \otimes e_{i_s}^{d-1}\qquad (i_s \notin J).$$
\end{remark}

\begin{obs}\label{obs:Lranksanddiffer}
Let $L_d^{i,B} (F)$ denote the submodule of Definition \ref{def:Lsubmodules}. Then the Koszul differential induces a map
$$L_d^{i,B} (F) \to L_d^{i-1,B} (F).$$
Moreover, if $\rk_i$ is as in the statement of Corollary \ref{cor:btablemaxidealiterated}, then
\begin{align*}
    \rank L_d^{i,B} (F) &= \rank L_d^i (F) - \rk_i \\
    &= \binom{n+d-1}{i+d} \binom{d+i-1}{i} - \rk_i. \\
\end{align*}
\end{obs}

\begin{proof}
The first observation is clear by noticing that $\cat{S}$ as in Definition \ref{def:Lsubmodules} generates $\ker \begin{pmatrix}
\phi^{I_1,\alpha^1}_i \\
\phi^{I_2,\alpha^2}_i \\
\vdots \\
\phi^{I_{r},\alpha^{r}}_i \\
\end{pmatrix}$, where each $\phi^{I_s,\alpha^s}_i$ is as in Definition \ref{def:generalqmaps}. The fact that this generates the kernel follows by the proof of Proposition \ref{prop:maxpowerqranks}. For the rank count, observe that the count for each omitted basis element is precisely the count done in the proof of Proposition \ref{prop:maxpowerqranks}. Indeed, the basis elements omitted are precisely the elements whose images form a basis for the image of the $q_i^s$ maps, for each $1 \leq s \leq r$.
\end{proof}

\begin{definition}
Let $\alpha := (\alpha_1 , \dots , \alpha_n)$ be an exponent vector. Define $\supp (\alpha) := \{ i \mid \alpha_i > 0 \}$. If $n_\alpha >1$, define $K^{\alpha^c}_\bullet$ to be the complex induced by the map
\begingroup\allowdisplaybreaks
\begin{align*}
    \psi: K_1^{\alpha} := \bigoplus_{i \notin \supp (\alpha)} Re_i &\to R \\
    e_i &\mapsto x_i .
\end{align*}
\endgroup
If $n_\alpha = 1$, then $K^{\alpha}_\bullet$ is defined to be the $0$ complex.
\end{definition}
It turns out that the following result tells us that the top linear strand of the minimal free resolution quotient defined by the ideals of Theorem \ref{thm:theMinlResn} will always be a direct sum of shifted Koszul complexes.

\begin{prop}\label{prop:easyCokernel}
Adopt notation and hypotheses as in Setup \ref{set:trimmingmaxideal}. Then there is an isomorphism of complexes
$$\Phi_\bullet : \coker \phi^{I,\alpha}_\bullet \to K_\bullet^{\alpha} [ -n_\alpha]$$
\end{prop}

\begin{proof}
If $n_\alpha = 1$, then the claim is true. Assume that $n_\alpha >1$. Observe that $\coker \phi^{I,\alpha}_i$ is free on all basis elements of the form
$$\{ e_{\supp (\alpha) \cup J } \mid \supp (\alpha) \cap J = \varnothing, \ |J| = i - n_\alpha \}.$$
Consider the map
\begingroup\allowdisplaybreaks
\begin{align*}
    \Phi_i : \coker \phi^\alpha_i &\to K^{\alpha}_{i- n_\alpha} \\
    e_{\{k_1 , \dots , k_{n_\alpha} \} \cup J } & \mapsto \sgn (J) e_J. 
\end{align*}
\endgroup
This map is clearly an isomorphism, whence it remains to show that $\Phi_\bullet$ is a morphism of complexes. For $i \geq 1$, consider the diagram
\begin{equation}\label{eq:diagram}
    \xymatrix{\coker \phi_i^{\alpha} \ar[d]^-{\Phi_i} \ar[r]^-{d_i} & \coker \phi_{i-1}^\alpha \ar[d]_-{\Phi_{i-1}} \\
K^{\alpha}_{i - n_\alpha} \ar[r]^-{\kos} & K^{\alpha}_{i-1-n_\alpha} \\}
\end{equation}
Going clockwise around \ref{eq:diagram}, one has:
\begingroup\allowdisplaybreaks
\begin{align*}
    e_{\{k_1 , \dots , k_{n_\alpha} \} \cup J } &\xmapsto{d_i} \sum_{i=1}^{n_\alpha} \sgn (k_i) x_{k_i} e_{\{k_1 , \dots , \widehat{k_i} , \dots k_{n_\alpha} \} \cup J } \\
    &+ \sum_{j \in J} \sgn( j \in \supp (\alpha) \cup J ) x_j e_{\{k_1 , \dots , k_{n_\alpha} \} \cup J \backslash j } \\
    &= \sum_{j \in J} \sgn( j \in \supp (\alpha) \cup J ) x_j e_{\{k_1 , \dots , k_{n_\alpha} \} \cup J \backslash j } \\
    &\xmapsto{\Phi_{i-1}}  \sum_{j \in J} \sgn( j \in \supp (\alpha) \cup J ) \sgn(J \backslash j  \subseteq \supp(\alpha) ) x_j e_{ J \backslash j }. \\
\end{align*}
\endgroup
where the equality in the penultimate line follows by noticing that $\im \phi_i^{I,\alpha}$ is free on basis elements of the form
$$\{ e_{J} \mid \supp(\alpha) \not\subset J, \ |J| = i \}.$$
Moving counterclockwise around \ref{eq:diagram}:
\begingroup\allowdisplaybreaks
\begin{align*}
    e_{\{k_1 , \dots , k_{n_\alpha} \} \cup J } &\xmapsto{\Phi_i} \sgn (J) e_J \\
    &\xmapsto{\kos} \sum_{j \in J } \sgn(J \subset \supp (\alpha) ) \sgn(j \in J) x_j e_{J \backslash j} . 
\end{align*}
\endgroup
To conclude, observe that
$$\sgn( j \in \supp (\alpha) \cup J ) \sgn(J \backslash j  \subseteq \supp(\alpha) ) =  \sgn(J \subset \supp (\alpha) ) \sgn(j \in J).$$
\end{proof}
Combining these building blocks with Theorem \ref{thm:splitMappingCone}, one obtains:

\begin{theorem}\label{thm:theMinlResn}
Adopt notation and hypotheses as in Setup \ref{set:trimmingmaxidealiter}. Let $K'$ be an equigenerated momomial ideal with $\overline{K'} = (x^{\alpha^1}, \dots , x^{\alpha^r})$. Then the minimal free resolution of $R / K'$ is given by the complex
$$F_i := L_d^{i-1,B} \oplus \Big( \bigoplus_{j=1}^{|B|} K_{i- n_{\alpha^j}}^{\alpha^{j}} \Big) \quad (i >0),$$
$$F_0 = R,$$
with differentials
$$\ell_i := \begin{pmatrix}
\kos^\psi \otimes 1 & \Theta_i \\
0 & - \bigoplus_{j=1}^{|B|} \kos^\psi \\
\end{pmatrix}, \quad (i >2),$$
$$\ell_2 := \begin{pmatrix}
\kos^\psi \otimes 1 & \Theta_i \\
\end{pmatrix},$$
$$\ell_1 := S(\psi)|_{L^{0,B}_d}$$
where $\Theta_p$ restricted to each direct summand $K^{\alpha^{s,c}}_{p-n_{\alpha^s}}$ is the map:
\begingroup\allowdisplaybreaks
\begin{align*}
    \Theta_i : K_{p- n_{\alpha^s}}^{\alpha^{j,c}} & \to L^{p-2,B}_d \\
    e_J^{\alpha^j} &\mapsto  \sgn (J) \sum_{i =1}^{n_{\alpha^s}} \sgn (k^s_i) x_{k^s_i}^2 e_{ \{ k^s_1 , \dots , \widehat{k^s_i} , \dots , k^s_{n_{\alpha^s}} \}  \cup J } \otimes e^{\alpha^s - \epsilon_{k^s_i}} \\
        &+ \sgn (J) \sum_{i<j} x_{k^s_i} x_{k^s_j} \big( \sgn(k^s_j) e_{ \{ k^s_1 , \dots , \widehat{k^s_j} , \dots , k^s_{n_{\alpha^s}} \}  \cup J } \otimes e^{\alpha^s - \epsilon_{k^s_i}} \\
        & \qquad \qquad \qquad + \sgn(k^s_i) e_{ \{ k^s_1 , \dots , \widehat{k^s_i} , \dots , k^s_{n_{\alpha^s}} \}  \cup J } \otimes e^{\alpha^s - \epsilon_{k^s_j}} \big) \\
\end{align*}
\endgroup
\end{theorem}

\begin{proof}
Using Theorem \ref{thm:splitMappingCone}, the proof comes down to computing the map $\Theta_i$ explicitly. For ease of notation/computation, assume that $|B|=1$. It will be understood that in the case $|B| >1$, this computation yields the restriction of $\Theta_p$ to each direct summand. 

Let $p \geq 2$; one computes the image of an arbitrary $e_J \in K^{\alpha^c}_{p - n_\alpha}$ under $\Theta_p$:
\begingroup\allowdisplaybreaks
\begin{align*}
    e_J &\xmapsto{\textrm{Prop} \ \ref{prop:easyCokernel}} \sgn (J) e_{ \{k_1 , \dots , k_{n_\alpha} \} \cup J} \\
        &\xmapsto{\kos^\psi} \sgn(J) \sum_{j \in J } \sgn (j ) x_j e_{\{k_1 , \dots , k_t \}\cup J} \\
        &+ \sgn (J) \sum_{i=1}^{n_{\alpha}} \sgn (k_i) x_{k_i} e_{ \{ k_1 , \dots , \widehat{k_i} , \dots , k_{n_\alpha} \} \cup J} \\
        &\xmapsto{\textrm{projection}} \sgn (J) \sum_{i=1}^{n_{\alpha}} \sgn (k_i) x_{k_i} e_{ \{ k_1 , \dots , \widehat{k_i} , \dots , k_{n_\alpha} \} \cup J} \\
        &\xmapsto{\textrm{inclusion}} \sgn (J) \sum_{i=1}^{n_{\alpha}}  x_{k_i} e_{\{ k_1 , \dots , k_{n_\alpha} \} \cup J} \otimes e^{\alpha - \epsilon_{k_i}} \\
        &\xmapsto{\kos^\psi \otimes 1} \sgn (J) \sum_{i=1}^{n_{\alpha}}  \sgn (j \in J) x_{k_i} x_j e_{ \{ k_1 , \dots , k_{n_\alpha} \}  \cup J \backslash j} \otimes e^{\alpha - \epsilon_{k_i}} \\
        &+ \sgn(J) \sum_{i,j =1}^{n_\alpha}  \sgn (k_j) x_{k_i} x_{k_j} e_{ \{ k_1 , \dots , \widehat{k_j} , \dots , k_{n_\alpha} \}  \cup J } \otimes e^{\alpha - \epsilon_{k_i}} \\ 
        &\xmapsto{\textrm{projection}} \sgn (J) \sum_{i,j =1}^{n_\alpha} \sgn (k_j) x_{k_i} x_{k_j} e_{ \{ k_1 , \dots , \widehat{k_j} , \dots , k_{n_\alpha} \}  \cup J } \otimes e^{\alpha - \epsilon_{k_i}} \\ 
        &= \sgn (J) \sum_{i =1}^{n_\alpha} \sgn (k_i) x_{k_i}^2 e_{ \{ k_1 , \dots , \widehat{k_i} , \dots , k_{n_\alpha} \}  \cup J } \otimes e^{\alpha - \epsilon_{k_i}} \\
        &+ \sgn (J) \sum_{i<j} x_{k_i} x_{k_j} \big( \sgn(k_j) e_{ \{ k_1 , \dots , \widehat{k_j} , \dots , k_{n_\alpha} \}  \cup J } \otimes e^{\alpha - \epsilon_{k_i}} \\
        & \qquad \qquad \qquad + \sgn(k_i) e_{ \{ k_1 , \dots , \widehat{k_i} , \dots , k_{n_\alpha} \}  \cup J } \otimes e^{\alpha - \epsilon_{k_j}} \big). \\
\end{align*}
\endgroup
The final equality of the above is written in terms of the basis elements of $L^{p-2,B}_d$
\end{proof}

As a Corollary, we obtain the previously mentioned fact that the minimal free resolution in the case that the complementary ideal consists of pure powers is obtained by simply restricting to the subcomplex $A_\bullet$ (with notation as in Theorem \ref{thm:splitMappingCone}).

\begin{cor}\label{cor:resnfordeletingpurepowers}
Adopt notation and hypotheses as in Setup \ref{set:trimmingmaxidealiter}, with $\alpha_s = d \epsilon_{k_s}$ for $1 \leq s \leq r \leq n$, where $B=\{ d \epsilon_{k_1} < \cdots < d \epsilon_{k_r} \}$. Let $K'$ be an equigenerated momomial ideal with $\overline{K'} = (x_{k_1}^d, \dots , x_{k_r}^d)$. Then the minimal free resolution of $R / K'$ is given by the complex
\begin{equation*}
    \begin{split}
        &L^B(\psi , d) : \xymatrix{0 \ar[r] & L_d^{n-1,B} \ar[rr]^-{\kos^\psi \otimes 1} & & \cdots \ar[rr]^{\kos^\psi \otimes 1} & & L_d^{0,B} \ar[r]^-{S_d (\psi)} & R \ar[r] & 0. } \\
    \end{split}
\end{equation*}
\end{cor}

\begin{proof}
This follows immediately by Theorem \ref{thm:splitMappingCone}, since $C_k = 0$ for all $k$ by Proposition \ref{prop:qrankforpurepower}.
\end{proof}

Next, we define the necessary building blocks in the squarefree case. This resolution will be much simpler to describe, since we have already seen that these ideals have linear minimal free resolutions. This implies that, as in Corollary \ref{cor:resnfordeletingpurepowers}, the minimal free resolution is obtained by taking the kernel of an appropriate morphism of complexes.

\begin{definition}\label{def:Fsubmodules}
Adopt notation and hypotheses as in Setup \ref{set:ittrimmingsqfree}, with $\I = \{ I_1 , \dotsc , I_r \}$. For each $i$, define the free submodule $F^{d,n,\I}_i \subseteq F^{d,n}_i$ to be generated by the following collections of basis elements, denoted $\cat{T}$ (all terms appearing are assumed to be standard tableau with strictly increasing columns and rows; recall that Notation \ref{notation:simplifyingTableaux} is in play here):
{\small
$$\begin{cases}
f_J \otimes f^{L} & \textrm{if} \ J \neq I_s \backslash \{ i_{ps} \} \ \textrm{for any} \ p, \ s, \\ 
 \\
f_J \otimes f^{I_s - \epsilon_{i_{ps}}} & \textrm{if} \  i_{ps} \notin J \ \textrm{for any} \ p, \ s, \\
\\
\sgn (i_{ps}) f_{J \cup \{ i_{ps} \}} \otimes f^{I_s - \epsilon_{i_{ps}}} + \sgn(i_{1s}) f_{  J \cup i_{1s}} \otimes f^{I_s - \epsilon_{i_{1s}}} & \textrm{where} \ I_s \cap J = \varnothing, \\
\end{cases}$$}
where $J = (j_0 < \dotsc < j_{i-1})$, $1 \leq s \leq r$, and $1 < p \leq d$.
\end{definition}

\begin{obs}\label{obs:Franksanddiffer}
Adopt notation and hypotheses as in Setup \ref{set:ittrimmingsqfree}, with $\I = \{ I_1 , \dotsc , I_r \}$. Let $F^{d,n,\I}_i$ denote the submodule of Definition \ref{def:Fsubmodules}. Then the differential $\partial_i^{d,n} : F_i^{d,n} \to F_{i-1}^{d,n}$ induces a differential
$$\partial_i^{d,n} : F_i^{d,n,\I} \to F_{i-1}^{d,n,\I}.$$
Moreover, 
\begin{align*}
    \rank F_i^{d,n,\I} &= \rank F_i^{d,n} - r \cdot \binom{n-d}{i-1} \\
    &= \binom{n}{d+i-1} \binom{d+i-2}{i-1} - r \binom{n-d}{i-1}. \\
\end{align*}
\end{obs}

\begin{proof}
The first claim follows after noting that $F_i^{d,n,\cat{I}}$ generates $\ker  \begin{pmatrix}
\psi^{I_1,I_1^c}_\ell \\
\psi^{I_2,I_2^c}_\ell \\
\vdots \\
\psi^{I_{r},I_r^c}_\ell \\
\end{pmatrix}$, where each $\psi^{I_s,I_s^c}_\ell$ is as in Definition \ref{def:spechtqmaps}. For the second claim, fix an indexing set $I_s = ( i_{1s} < \cdots < i_{ds} )$. The module $F^{d,n,\cat{I}}_i$ omits precisely all standard basis elements of the form
$$\ytableausetup{boxsize=2em}\begin{ytableau}
\vdots & i_{2s} & \cdots & i_{ds} \\
J' \\
\vdots \\
\end{ytableau},$$
where $J' = (j_0' < \cdots < j_{i-1}' )$ and $J' \cap I_s = \{ i_{1s} \}$; there are $\binom{n-d}{i-1}$ such choices for $J'$ and $r$ choices of $s$, so the result follows. 
\end{proof}

\begin{theorem}\label{thm:resofsqfreetrimmed}
Adopt notation and hypotheses as in Setup \ref{set:ittrimmingsqfree}, with $\I = \{ I_1 , \dotsc , I_r \}$. Let $K'$ be a squarefree equigenerated momomial ideal with $\overline{K'} = (x^{I_1} , \dots , x^{I_r} )$. Then the minimal free resolution of $R / K'$ is given by the complex
\begin{equation*}
    \begin{split}
        &F^{d,n,\I}_\bullet : \xymatrix{0 \ar[r] & F^{d,n,\I}_{n-d+1} \ar[rr]^-{\partial_{n-d}^{d,n}} & & \cdots \ar[rr]^{\partial_1^{d,n}} & &F^{d,n,\I}_1 \ar[r] & R \ar[r] & 0. } \\
    \end{split}
\end{equation*}
\end{theorem}



\begin{proof}
This follows immediately by combining Theorem \ref{thm:splitMappingCone} with Proposition \ref{prop:spechtqranks}.
\end{proof}


We conclude with some questions about additional structure on the complexes above. Firstly, it is well known the the $L$-complexes admit the structure of an associative DG-algebra. Likewise, the complex constructed by Galetto will also admit the structure of an associative DG-algebra, since the squarefree Elihou-Kervaire complex admits such a structure by work of Peeva (see \cite{peeva19960}). One is then tempted to ask:
\begin{question}
Do the complexes of Theorem \ref{thm:theMinlResn} or Theorem \ref{thm:resofsqfreetrimmed} admit the structure of an associative DG-algebra?
\end{question}
Similarly, it is well known that the (squarefree) Eliahou-Kervaire resolution is cellular by work of Mermin (see \cite{mermin2010eliahou}). Since one can reformat the above constructions of this section in terms of taking kernels/cokernels of the Eliahou-Kervaire resolution, we also pose:
\begin{question}
Are the complexes of Theorem \ref{thm:theMinlResn} or Theorem \ref{thm:resofsqfreetrimmed} cellular?
\end{question}

\bibliographystyle{amsplain}
\bibliography{biblio}

\providecommand{\bysame}{\leavevmode\hbox to3em{\hrulefill}\thinspace}
\providecommand{\MR}{\relax\ifhmode\unskip\space\fi MR }
\providecommand{\MRhref}[2]{%
  \href{http://www.ams.org/mathscinet-getitem?mr=#1}{#2}
}
\providecommand{\href}[2]{#2}
\begin{thebibliography}{10}

\bibitem{aramova1998squarefree}
Annetta Aramova, J{\"u}rgen Herzog, and Takayuki Hibi, \emph{Squarefree
  lexsegment ideals}, Mathematische Zeitschrift \textbf{228} (1998), no.~2,
  353--378.

\bibitem{bayer1998cellular}
Dave Bayer and Bernd Sturmfels, \emph{Cellular resolutions of monomial
  modules}, Journal f{\"u}r die reine und angewandte Mathematik \textbf{1998}
  (1998), no.~502, 123--140.

\bibitem{buchsbaum1975generic}
David~A Buchsbaum and David Eisenbud, \emph{Generic free resolutions and a
  family of generically perfect ideals}, Advances in Mathematics \textbf{18}
  (1975), no.~3, 245--301.

\bibitem{eagon2019minimal}
John Eagon, Ezra Miller, and Erika Ordog, \emph{Minimal resolutions of monomial
  ideals}, arXiv preprint arXiv:1906.08837 (2019).

\bibitem{el2014artinian}
Sabine El~Khoury and Andrew~R Kustin, \emph{Artinian gorenstein algebras with
  linear resolutions}, Journal of Algebra \textbf{420} (2014), 402--474.

\bibitem{eliahou1990minimal}
Shalom Eliahou and Michel Kervaire, \emph{Minimal resolutions of some monomial
  ideals}, Journal of Algebra \textbf{129} (1990), no.~1, 1--25.

\bibitem{fulton1997young}
William Fulton, \emph{Young tableaux: with applications to representation
  theory and geometry}, vol.~35, Cambridge University Press, 1997.

\bibitem{galetto2020ideal}
Federico Galetto, \emph{On the ideal generated by all squarefree monomials of a
  given degree}, Journal of Commutative Algebra \textbf{12} (2020), no.~2,
  199--215.

\bibitem{gasharov2002resolutions}
Vesselin Gasharov, Takayuki Hibi, and Irena Peeva, \emph{Resolutions of
  a-stable ideals}, Journal of Algebra \textbf{254} (2002), no.~2, 375--394.

\bibitem{ha2006resolutions}
Huy~T{\`a}i H{\`a} and Adam Van~Tuyl, \emph{Resolutions of square-free monomial
  ideals via facet ideals: a survey}, arXiv preprint math/0604301 (2006).

\bibitem{mermin2010eliahou}
Jeffrey Mermin, \emph{The eliahou-kervaire resolution is cellular}, Journal of
  Commutative Algebra \textbf{2} (2010), no.~1, 55--78.

\bibitem{miller2018free}
Claudia Miller and Hamidreza Rahmati, \emph{Free resolutions of artinian
  compressed algebras}, Journal of Algebra \textbf{497} (2018), 270--301.

\bibitem{miller2004combinatorial}
Ezra Miller and Bernd Sturmfels, \emph{Combinatorial commutative algebra}, vol.
  227, Springer Science \& Business Media, 2004.

\bibitem{peeva19960}
Irena Peeva, \emph{0-borel fixed ideals}, Journal of Algebra \textbf{184}
  (1996), no.~3, 945--984.

\bibitem{reisner1976cohen}
Gerald~Allen Reisner, \emph{Cohen-macaulay quotients of polynomial rings},
  Advances in Mathematics \textbf{21} (1976), no.~1, 30--49.

\bibitem{taylor1966ideals}
Diana~Kahn Taylor, \emph{Ideals generated by monomials in an r-sequence,
  proquest llc}, Ann Arbor, MI (1966).

\bibitem{vandebogert2021resolution}
Keller VandeBogert, \emph{Resolution and tor algebra structures of grade 3
  ideals defining compressed rings}, Journal of Algebra \textbf{586} (2021),
  140--153.

\bibitem{vandebogert2021trimming}
\bysame, \emph{Trimming complexes and applications to resolutions of
  determinantal facet ideals}, Communications in Algebra \textbf{49} (2021),
  no.~3, 1017--1036.

\bibitem{weyman2003}
Jerzy Weyman, \emph{Cohomology of vector bundles and syzygies}, vol. 149,
  Cambridge University Press, 2003.

\end{thebibliography}
\addcontentsline{toc}{section}{Bibliography}

\end{document}